\documentclass[12pt]{amsart}

\input xy
\usepackage[english]{babel}
\usepackage[latin1]{inputenc}
\usepackage{graphicx}
\usepackage{amsmath}
\usepackage{amssymb}
\usepackage{amscd}
\usepackage{color}
\usepackage{oldgerm}
\usepackage{amsfonts}
\usepackage{newlfont}
\usepackage{longtable}
\usepackage{multirow}
\usepackage[dvipsnames]{xcolor}

\DeclareOldFontCommand{\rm}{\normalfont\rmfamily}{\mathrm}

\usepackage[all]{xy}

\usepackage{upref}

\def\F{\mathbb F}

\def\G{\mathcal G}

\def\ad{\operatorname{ad}}

\def\Der{\operatorname{Der}}

\def\dim{\operatorname{dim}}

\def\Ker{\operatorname{Ker}}

\def\Id{\operatorname{Id}}

\def\Span{\operatorname{Span}}

\def\Im{\operatorname{Im}}

\def\g{\mathfrak g}
\def\gl{\mathfrak{gl}}
\def\h{\mathfrak h}

\def\a{\mathfrak{a}}
\def\G{\mathfrak{G}}

\theoremstyle{plain}\swapnumbers

\newtheorem{Theorem}{Theorem}[section]

\newtheorem{Prop}[Theorem]{Proposition}

\newtheorem{Cor}[Theorem]{Corollary}

\newtheorem{Remark}[Theorem]{Remark}

\title[]{On quadratic Hom-Lie algebras 
with twist maps in their centroids
and their relationship with quadratic Lie algebras}

\author[R. Garc\'ia-Delgado, G. Salgado, O.A. S\'anchez-Valenzuela]{R. Garc\'ia-Delgado$^{(1)}$, G. Salgado$^{(2)}$
and O.A. S\'anchez-Valenzuela$^{(1)}$}

\address{(1) Centro de Investigaci\'on en Matem\'aticas, A.C., Unidad M\'erida;
Yucat\'an CP 97302, M\'exico.}
\email{rosendo.garcia@cimat.mx, rosendo.garciadelgado@alumnos.uaslp.edu.mx}
\email{adolfo@cimat.mx}

\address{(2) Fac. de Ciencias, UASLP, 
Av. Parque Chapultepec 1570, Priv. del Pedregal,
San Luis Potos\'{\i}, S.L.P. CP 78295 M\'exico.}
\email{gsalgado@fciencias.uaslp.mx, gil.salgado@gmail.com}


\date{\today}

\begin{document}

\begin{abstract}
Hom-Lie algebras having non-invertible twist maps
in their centroids are studied. 
Central extensions of Hom-Lie algebras having these properties
are obtained and shown how the same properties are preserved. 
Conditions are given so that the produced central extension 
has an invariant metric with respect to its Hom-Lie product 
making its twist map self-adjoint when the original 
Hom-Lie algebra has such a metric. 
This work is focused on algebras with these properties and 
following Benayadi and Makhlouf
we call them quadratic Hom-Lie algebras. 
It is shown how a quadratic Hom-Lie algebra gives rise to a quadratic 
Lie algebra and that the Lie algebra associated to the given Hom-Lie
central extension is a Lie algebra central extension of it. 
It is also shown that if the
Hom-Lie product is not a Lie product, there exists 
a non-abelian algebra, which is in general non-associative too,
the commutator of whose product is precisely the Hom-Lie product
of the Hom-Lie central extension. 
Moreover, the algebra whose commutator realizes
this Hom-Lie product is shown to be simple if the
associated Lie algebra is nilpotent.
Non-trivial examples are provided.
\end{abstract}

\maketitle

\section*{Introduction}

\medskip
All algebras and vector spaces in this work
are assumed to be finite dimensional and defined 
over an algebraically closed field $\F$ of characteristic zero.

\medskip
A {\it Hom-Lie algebra\/} over a field $\F$
is a triple $(\g,\mu,T)$ consisting of a vector space $\g$,
a skew-symmetric bilinear product $\mu:\g\times\g\to\g$
and a linear map $T:\g\to\g$, called {\it the twist map\/,} satisfying:
\begin{equation}\label{primera Hom-Lie Jacobi identity}
\mu(T(x),\mu(y,z))+\mu(T(y),\mu(z,x))+\mu(T(z),\mu(x,y))=0,
\end{equation}
for all $x,y$ and $z$ in $\g$. 

\medskip
The algebraic study of Hom-Lie algebras is usually traced back to the
work of Hartwig,  Larsson  and  Silvestrov (see \cite{HLS})
in studying deformations of the Witt and Virasoro algebras.
We shall refer our readers to the references provided
by Benayadi and Makhlouf in \cite{Ben}, 
and specially to the Introduction of their work, 
where a rather complete and comprehensive
summary was given for the various classes of Hom-structures that had
been studied, at least up to 2014 when \cite{Ben} was written.
Since then, the study of Hom-structures has grown fastly, covering
a wide range of new interesting issues
(see for example \cite{Cai}, \cite{MandZ}, \cite{MandN}, 
\cite{MandS}, \cite{WendeLiu}, or \cite{Z}).

\medskip
The spirit of our work is similar to that of \cite{Ben},
in that we also construct quadratic Hom-Lie algebras
and study their connections with quadratic Lie algebras. 
The main differences are that we do not start with a Lie algebra $(\g,\mu)$, 
and our results are focused on {\it non-invertible twist maps\/} $T$
which are furthermore {\it centroids of\/} $(\g,\mu)$
(see for example \cite{Luks} and \cite{Melville}); that is,
we shall assume throughout this work that,
\begin{equation}\label{propiedad de equivariancia en Hom-Lie algebras}
T\left(\mu(x,y)\right)=\mu\left(T(x),y\right),\quad \text{ for all }\ x,y \in \g.
\end{equation}
It is immediate to show that if $T$ is an invertible centroid,
then $\mu$ is a Lie bracket. That is why we shall assume that $\dim\Ker(T)>0$, 
and when be referred to a pair $(\g,\mu)$, the bilinear map 
$\mu:\g\times\g\to\g$ will only be assumed to be skew-symmetric.

\medskip
We are also interested in Hom-Lie algebras $(\g,\mu,T)$ 
with $T$ satisfying the hypotheses above 
and $\g$ equipped with a symmetric, 
non-degenerate bilinear form $B:\g \times \g \to \F$, 
satisfying,
$$
\aligned
B(\mu(x,y),z) 
& =
B(x,\mu(y,z)),
\qquad\text{and}\\
B(T(x),y)
& =
B(x,T(y)),
\quad\qquad\text{for all\ }x,y,z\in \g.
\endaligned
$$
Such a $B$ will be referred to as {\it a $\mu$-invariant metric\/} on $(\g,\mu,T)$
and, following \cite{Ben}, we shall say that $(\g,\mu,T, B)$ is a {\it quadratic Hom-Lie algebra\/.} 
If $T=\operatorname{Id}_{\g}$, such a $4$-tuple $(\g,\mu,T,B)$ is just a {\it quadratic Lie algebra\/.}
It is shown in \cite{Ben} that the notion of double extension for Lie algebras
can be generalized and carried out on quadratic Hom-Lie algebras
whose twist map is an invertible morphism.
By way of contrast,
what we do in this work is to characterize the quadratic Hom-Lie algebras 
satisfying the more relaxed hypotheses just described above for the twist map
$T$ and the $\mu$-invariant form $B$, using central extensions of quadratic Lie algebras.

\medskip
A few words are worth our while to answer the question of, {\it why looking
at quadratic Hom-Lie algebras with non-invertible centroids as twist maps?\/}
It is known that the space of self-adjoint centroids for a $\mu$-invariant metric $B$
on $\g$ plays an important role in the theory of quadratic Lie algebras, since
any pair of invariant metrics are related by an invertible self-adjoint centroid.
Also, Schur's Lemma states that for simple Lie algebras, the 
space of self-adjoint centroids is one-dimensional, which leads to the well known
fact that any invariant metric on a simple Lie algebra must be a non-zero scalar
multiple of the Cartan-Killing form. 
In \cite{Benayadi}, quadratic Lie algebras are studied under the hypothesis
that its space of self-adjoint centroids is 2-dimensional and some interesting 
consequences of this fact are obtained there. For example, it is proved
that all such Lie algebras are {\it local\/;} that is, they have a unique
maximal ideal.
Thus, after our findings in \cite{GSV},
it was natural for us to try to understand some of the characteristics
shown by quadratic Hom-Lie algebras under the more general conditions
assumed for them in this work.

\medskip
Here is now a more detailed
description of what we do here.
We show first in \textbf{Prop.\,\ref{proposicion cero}}
and \textbf{Thm.\,\ref{teorema nuevo 2}}
that a quadratic Hom-Lie algebra $(\g,\mu,T,B)$
having $T$ in its centroid
and $\Ker(T)\ne\{0\}$, 
gives rise to a quadratic Lie algebra
$(\G,[\,\cdot\,,\,\cdot\,],B_{\G})$ which is in turn a central extension
by $V$ of a quadratic Lie algebra $(\g,[\,\cdot\,,\,\cdot\,],B)$
having $[\,\cdot\,,\,\cdot\,]=T\circ\mu$, 
making $V$ isotropic under $B_{\G}$
and $\dim V=\dim\Ker(T)$.
In \textbf{Prop.\,\ref{teorema nuevo 3}} we show how to
extend the original Hom-Lie algebra $(\g,\mu,T)$
to a Hom-Lie algebra $(\G,\mu_{\G},L)$. A fundamental role
is played by the maps $k:\g\to\G$ and $h:\G\to\g$
defined in \eqref{definicion de k} and \eqref{definicion h}, respectively.
The result also makes use of the Lie algebra invariant metric
$B_{\G}$ and shows that: (1) the twist map $L$ is $B_{\G}$-self-adjoint,
(2) $L$ restricted to $\g$ is exactly $k$, and (3) the projection map $\G\to\g$
that mods out the central ideal $V$ actually yields a
Hom-Lie algebra epimorphism $(\G,\mu_{\G},L)\to(\g,\mu,T)$.
It is then natural to say that $(\G,\mu_{\G},L)$
\textbf{is a central extension of $(\g,\mu,T)$ as Hom-Lie algebras},
since $\mu_{\G}(\G,V)=\{0\}$ and $L$ vanishes on $V$
as shown in \textbf{Prop.\,\ref{teorema nuevo 3}}.
As a consequence of these results we show in
\textbf{Cor.\,\ref{corolario cociclos}}
that if $B_{\G}$ is $\mu_{\G}$-invariant or 
if $L$ is a centroid for $\mu_{\G}$,
then the $2$-cocycle used to define the central extension
$\G=\g\oplus V$ has to be identically zero. It is also proved that if
$\mu$ is not a Lie product,
then $\Ker(T)$
is equal to the center of the Lie algebra $(\g,[\,\cdot\,,\cdot\,])$ 
(see \textbf{Prop.\,\ref{Prop cociclos}}).

\medskip
At this point we make a general observation in
\textbf{Prop.\,\ref{construccion}} that produces a quadratic Hom-Lie algebra
$(\g,\mu,T,B)$ on the underlying space $\g=\h\oplus\h^*$, 
with $T$ in its centroid, when a Hom-Lie algebra $(\h,\nu,S)$
is given, having $S$ in its centroid.
In fact, $T(x+\alpha)=S(x)+S^*(\alpha)$, and 
$B(x+\alpha,y+\beta)=\alpha(y)+\beta(x)$, 
for any $x$ and $y$ in $\h$ and any $\alpha$ and $\beta$ in $\h^*$. 
With this construction we are able to exhibit an example of a 
$6$-dimensional Hom-Lie algebra  $(\h,\nu,S)$ for which 
$\nu$ is not a Lie bracket and we also describe explicitly the 
$12$-dimensional quadratic Hom-Lie algebra $(\g,\mu,T,B)$ 
based on $\g=\h\oplus\h^*$.

\medskip
In \textbf{Thm.\,\ref{teorema nuevo 2}}
we start from a quadratic Hom-Lie algebra 
$(\g,\mu,T,B)$, with $T$ in its centroid and $\Ker T\ne\{0\}$, to
produce the quadratic Lie algebra 
$(\G=\g\oplus V,[\,\cdot\,,\,\cdot\,]_{\G},B_{\G})$ 
where 
$V$ is isotropic, $[\G,V]_{\G}=\{0\}$
and $\dim V=\dim\Ker(T)$. 
Then, in \textbf{Prop.\,\ref{proposicion conexion}} we show 
that there is a bilinear product in $\G$,
$(x,y)\mapsto xy$ satisfying,
$$
\aligned
\mu_{\G}(x,y)&=xy-yx,\qquad\text{and},
\\
2xy &= \mu_{\G}(x,y)-[h(x),y]_{\G}+[x,h(y)]_{\G}.
\endaligned
$$
\medskip
Now, one of the goals
that originally motivated this work was to
construct (and if possible, to characterize) quadratic Hom-Lie 
algebras with
twist maps in their centroids
{\it when the underlying Hom-Lie product is not a Lie algebra bracket\/.}  
Using the product $(x,y)\mapsto xy$ in $\G$ above
we attain this goal in  \textbf{Thm.\,\ref{teorema}} through the
product defined on $\mathcal{A}=\mathbb F \times\G$ by,
$$
(\xi,x)\,(\eta,y)=(\,\xi\, \eta+B_{\G}(x,y)\,,\,\xi\, y+\eta\, x+x\,y\,),
$$
for all $\xi,\eta$ in $\F$ and all $x,y$ in $\G$.
It turns out that
$\mathcal{A}$ is, in general, a non-associative algebra
with unit element $1_{\mathcal{A}}=(1,0)$,
and if $\mu$ is not a Lie product
then $\mathcal{A}$  has no non-trivial two-sided ideals.
Moreover, $\mathcal{A}$ turns out to be simple
when the Lie algebra $(\g,[\,\cdot\,,\,\cdot\,])$ is nilpotent.
It is also shown  in  \textbf{Thm.\,\ref{teorema}}
that the induced product on 
$\mathcal{A}^{\prime}=\mathcal{A}/\mathbb F\,1_{\mathcal{A}}$
is given by,
$$
\mu^{\prime}\left((\xi,x)+\F 1_{\mathcal{A}},(\eta,y)+\F 1_{\mathcal{A}}\right)=
(\xi,x)(\eta,y)-(\eta,y)(\xi,x)+\F\, 1_{\mathcal{A}},
$$
for all $\xi,\eta$ in $\F$ and $x,y$ in $\G$. 
Then, defining
$T^{\prime}:\mathcal{A}^{\prime} \to \mathcal{A}^{\prime}$ 
by  $T^{\prime}((\xi,x)+\F 1_{\mathcal{A}}))=(\xi,L(x))+\F 1_{\mathcal{A}}$,
one obtains a  Hom-Lie algebra isomorphism
$\left(\mathcal{A}^{\prime},\mu^{\prime},T^{\prime}\right)\to\left(\G,\mu_{\G},L \right)$, 
into the Hom-Lie algebra of \textbf{Prop.\,\ref{teorema nuevo 3}},
thus also recovering the Hom-Lie algebra structure on $(\g,\mu,T)$.

\medskip
To round up our results
we conclude by showing
in \textbf{Thm.\,\ref{teorema extensiones centrales}} 
---using some results from \cite{GSV}---
that if one starts with a quadratic Lie algebra 
$(\g,[\,\cdot\,,\,\cdot\,],B)$ and a central extension
$(\G=\g\oplus V,[\,\cdot\,,\,\cdot\,]_{\G})$
with $\dim V>0$ that admits an invariant metric $B_{\G}$
under which $V$ is isotropic, then there exists a
quadratic Hom-Lie algebra $(\g,\mu,T,B)$
with $T$ in its centroid such that,
$[\,\cdot\,,\,\cdot\,]=T\circ\mu$ and $\dim\Ker(T)=\dim V$.

\medskip
We close this work giving a non-trivial example of how to start
with a $6$-dimensional quadratic Lie algebra $(\g,[\,\cdot\,,\,\cdot\,],B)$, then
show how to produce the $9$-dimensional 
Lie algebra central extension
$(\G=\g\oplus V,[\,\cdot\,,\,\cdot\,]_{\G}, B_{\G})$
having $\dim V=3$,
and finally how to produce the $6$-dimensional quadratic Hom-Lie algebra
$(\g,\mu,T,B)$ with $T$ 
in its centroid,
and the $9$-dimensional
Hom-Lie algebra $(\G,\mu_{\G},L)$ with the help of the simple algebra $\mathcal{A}$.

\section{
Characterization of Hom-Lie algebras with twist maps in their centroids}

\medskip
We start by showing how to obtain a quadratic Lie algebra
from a quadratic
Hom-Lie algebra $(\g, \mu, T, B)$ having
$T$ in the centroid.

\medskip
\begin{Prop}\label{proposicion cero}{\sl
Let $(\g, \mu, T, B)$ be a quadratic Hom-Lie algebra
with $T$ in its centroid.
Let $[\,\cdot\,,\,\cdot\,]:\g\times\g\to\g$ be defined by $[x,y]=T\left(\mu(x,y)\right)$, 
for all $x$ and $y$ in $\g$. 
Then, $(\g,[\,\cdot\,,\,\cdot\,],B)$ is a quadratic Lie algebra 
and $T([x,y])=[T(x),y]$, for all $x$ and $y$ in $\g$\/.}
\end{Prop}

\medskip
\begin{proof}
Define $[\,\cdot\,,\,\cdot\,]$ by means of $T\circ\mu$.
Since $T$ is a centroid,
$T([x,y])=T(\mu(T(x),y))=[T(x),y]$. 
It also follows that, 
$$
\left[x,[y,z]\right]=T\left(\mu\left(T(x),\mu(y,z)\right)\right),
$$
for any $x$, $y$ and $z$ in $\g$.
Then, the Jacobi identity for $[\,\cdot\,,\,\cdot\,]$ follows from 
 \eqref{primera Hom-Lie Jacobi identity} and 
the linearity of $T$. 
Since $T$ is $B$-self-adjoint,
the $\mu$-invariance of $B$ leads to
$B\left([x,y],z\right)=B\left(x,[y,z]\right)$,
proving that $(\g,[\,\cdot\,,\,\cdot\,],B)$ is a quadratic Lie algebra.
\end{proof}

\medskip
\textbf{Observe} that the construction of the Lie algebra
$(\g,[\,\cdot\,,\,\cdot\,])$ in \textbf{Prop.\,\ref{proposicion cero}} 
is similar to the one given in \cite{Ben} (\textbf{Prop.\,1.8} there).

\medskip
Now, the hypothesis $\dim\Ker(T)=r>0$ 
makes it possible to associate another Hom-Lie algebra to the 
quadratic Hom-Lie algebra $\left(\g,\mu,T,B\right)$;
say, $(\G,\mu_\G,L)$. Our first result in that direction shows
how to build up the underlying vector space $\G$ 
and how it actually has the structure of a quadratic Lie algebra.

\medskip
\begin{Theorem}\label{teorema nuevo 2}{\sl
Let $\left(\g, \mu, T, B\right)$ be 
a quadratic Hom-Lie algebra with $T$ 
in its centroid
and $\dim\Ker(T)=r>0$. 
For any $r$-dimensional vector space $V$,
$\G=\g \oplus V$ inherits a quadratic Lie algebra structure 
$\left(\G,[\,\cdot\,,\,\cdot\,]_{\G},B_{\G} \right)$, such that,

\medskip
\begin{itemize}
\item [\textbf{(i)}] 
$V$ is contained in its center $C(\G)$, and

\medskip
\item [\textbf{(ii)}] 
$B_{\G}$ makes $V$ isotropic; {\it ie\/,} for all $u$ and $v$ in $V$, $B_{\G}(u,v)=0$\/.
\end{itemize}
}
\end{Theorem}

\medskip
\begin{proof}
Let $\a \subset \g$ be a complementary subspace to $\Im(T)$. 
Then, $\dim\a=\dim\Ker(T)=r$.
Let $\{\,a_1,\ldots,a_r\,\mid\,a_i\in\g\,\}$ be a basis of $\a$
and let $\{v_1,\ldots,v_r\,\mid\,v_i\in V\,\}$ be a basis of $V$.

\medskip
\textbf{Claim A1.} {\it There are: (1) an injective linear map $k:\g \to \G$, 
and, (2) a skew-symmetric bilinear map $[\,\cdot\,,\,\cdot\,]_{\G}$ in $\G$,
such that $k\left([x,y]\right)=\left[k(x),y\right]_{\G}$ for all $x$ and $y$ in $\g$.}

\medskip
Indeed. Define the linear map $k:\g \to \G$ by,
\begin{equation}\label{definicion de k}
k(x)=T(x)+B(a_1,x)v_1+\ldots+B(a_r,x)v_r,\quad \text{ for all }x \in \g.
\end{equation}
This is injective because $B$ is non-degenerate, $T$ is $B$-self-adjoint, 
and $\Ker(T) \cap \a^{\perp}=\{0\}$.

\medskip
Let $[\,\cdot\,,\,\cdot\,]_{\G}:\G \times \G \to \G$ be the bilinear map defined by:
\begin{equation}\label{corchete de extension central-1}
\aligned
\,[x,y]_{\G}&=k\left(\mu(x,y)\right),\qquad\text{ for all }x,y \in \g,
\\
\,[x+u,v]_{\G}&=0,\qquad\qquad\qquad\text{ for all }x \in \g\ \text{ and }u,v \in V.
\endaligned
\end{equation}
Since $k(x)-T(x)$ belongs to $V$ for any $x$ in $\g$, 
and since $[V,\G]_{\G}=\{0\}$, it follows 
from \eqref{corchete de extension central-1}
that for any $x$ and $y$ in $\g$,
\begin{equation}
k([x,y])=k(\mu(T(x),y))=[T(x),y]_{\G}=[k(x),y]_{\G}\,.
\end{equation}

\medskip
\textbf{Claim A2.} {\it $[\,\cdot\,,\,\cdot\,]_{\G}$ is a Lie bracket in $\G$\/.}

\medskip
Let $x,y,z$ be in $ \g$ and $u,v,w$ be in $V$. By \textbf{Claim A1} and \eqref{corchete de extension central-1} we have:
$$
\aligned
\,\left[x+u,[y+v,z+w]_{\G}\right]_{\G}&=\!\left[x,[y,z]_{\G}]_{\G}\!=\![x,k(\mu(y,z))\right]_{\G}\\
\,&=k\left(\mu(x,T(\mu(y,z)))\right)=k\left(\mu\left(T(x),\mu(y,z)\right)\right)\\
\endaligned
$$
Then the Hom-Lie Jacobi identity for $\mu$ implies that 
$[\,\cdot\,,\,\cdot\,]_{\G}$ satisfies the Jacobi identity. 
Thus, $\left(\G,[\,\cdot\,,\,\cdot\,]_{\G}\right)$ is a $(n+r)$-dimensional Lie algebra.

\medskip
\textbf{Claim A3.}
{\it There exists a non-degenerate symmetric bilinear form 
$B_{\G}$ on 
$\G=\g \oplus V=\a \oplus \Im T \oplus V$
under which $\g^{\perp}=\a$ and $V$ is isotropic\/.}

\medskip
Let $B_{\G}:\G \times \G \to \F$ be defined  by,
\begin{equation}\label{definicion de metrica en extension central}
\begin{split}
& B_{\G}\left(T(x),T(y)\right)=B(T(x),y),\quad \text{ for all }x,y \in \g,\\
& B_{\G}\left(\a \oplus V,\Im(T)\right)=B_{\G}(\a,\a)=B_{\G}(V,V)=\{0\},\\
& B_{\G}\left(a_i,v_j\right)=\delta_{ij},\quad\quad \text{ for all }1 \leq i,j \leq r.
\end{split}
\end{equation}
Clearly $B_{\G}$ is non-degenerate. 
Since $\g=\a \oplus \Im(T)$, it follows that,
\begin{itemize}
\item[\textbf{(P1)}] $\g^{\perp}=\{x \in \G\mid B_{\G}(x,y)=0,\,\text{ for all }y \in \g\}=\a$;
\medskip
\item[\textbf{(P2)}] $\Im(T) \subset V^{\perp}$ in $\G$.
\end{itemize}

\medskip
\textbf{Claim A4.} {\it $\Im k=V^{\perp}$\/.}

\medskip
It follows from \eqref{definicion de metrica en extension central}
that $V^{\perp}=\Im(T) \oplus V$. 
From the definition of $k$ (see \eqref{definicion de k}), 
it is clear that $B_{\G}\left(\Im(k),V\right)=\{0\}$. Thus, $\Im(k) \subset V^{\perp}$. 
Since $k$ is injective and $\dim\Im(k)=\dim V^{\perp}$, 
we conclude that $\Im(k)=V^{\perp}$. 

\medskip
\textbf{Claim A5.} {\it The following relation between $B$ and $B_{\G}$ holds true\/:}
\begin{equation}\label{relacion entre las metricas con k}
B(x,y)=B_{\G}(k(x),y+u),\,\,\text{ for all }x,y \in \g \text{ and any }\,u \in V.
\end{equation}
\medskip
To prove \eqref{relacion entre las metricas con k} we recall
first from \eqref{definicion de metrica en extension central} that
$B_{\G}\left(a_i,v_j\right)=\delta_{ij}$ for $1\!\leq i,j\!\leq r$.
Now, using the definition of $k$ (see \eqref{definicion de k}) 
and the fact that $\a=\g^\perp$ in $\G$ 
(\textbf{(P1)} above),
we see that for any $y$ in $\g$,
$$
B_{\G}(a_i,k(y))=B_{\G}\left(a_i,T(y)+B(a_1,y)v_1+\cdots+B(a_r,y)v_r\right)=B(a_i,y),
$$
for all $1 \leq i \leq r$. 
Therefore,
\begin{equation}\label{a1}
B_{\G}(a,k(y))= B(a,y), \,\,\text{ for all } a \in \a \text{ and }y \in \g.
\end{equation}
Now, for $x$ in $\g$ we know that
$k(x)-T(x) \in V \subset V^{\perp}$ (see \eqref{definicion de k}).
By \textbf{(P2)} we also know that $\Im(T) \subset V^{\perp}$.
Then, for any $x$ and $y$ in $\g$ we have,
\begin{equation}\label{a2}
B_{\G}(T(x),k(y))=B_{\G}(T(x),T(y))=B(T(x),y).
\end{equation}
Combining \eqref{a1} and \eqref{a2}, we conclude that 
$B_{\G}(a+T(x),k(y))=B(a+T(x),y)$,
for any $x$ and $y$ in $\g$ and $a$ in $\a$.
Since $\g=\a \oplus \Im(T)$ we deduce that $B(x,y)=B_{\G}(x,k(y))$ for all $x$ and $y$ in $\g$. 
The symmetry of both $B$ and $B_{\G}$ leads to $B(x,y)=B_{\G}(k(x),y)=B_{\G}(x,k(y))$. 
Finally, from \textbf{Claim A4} we have $\Im(k)=V^{\perp}$. 
Therefore, $B(x,y)=B_{\G}(k(x),y+u)$, for all $x,y$ in $\g$ and $u$ in $V$.

\medskip
\textbf{Claim A6.} {\it $B_{\G}$ is invariant under $[\,\cdot\,,\,\cdot\,]_{\G}$\/.}

\medskip
Let $x,y,z$ be in $\g$ and $u,v,w$ be in $V$. 
We shall prove first that $B_{\G}\left([x+u,y+v]_{\G},z+w\right)=B\left(\mu(x,y),z\right)$. In fact,
$$
\aligned
B_{\G}\left([x+u,y+v]_{\G},z+w\right)
&=
\!\!B_{\G}\left([x,y]_{\G},z+w\right)
\!=\!\!
B_{\G}\left(k(\mu(x,y)),z+w\right)\\
\,&=B_{\G}\left(k(\mu(x,y)),z\right)=B(\mu(x,y),z),
\endaligned
$$
where use has been made of
$V \subset C(\G)$, $k \left(\mu(x,y)\right)=[x,y]_{\G}$ 
(see \eqref{corchete de extension central-1}), and $\Im(k)=V^{\perp}$,
together with \eqref{relacion entre las metricas con k}.
A similar argument proves that $B_{\G}\left(x+u,[y+v,z+w]_{\G}\right)=B(x,\mu(y,z))$. 
Finally, the $\mu$-invariance of $B$ implies that $B_{\G}$ is invariant under 
$[\,\cdot\,,\,\cdot\,]_{\G}$. Therefore $\left(\G,[\,\cdot\,,\,\cdot\,],B_{\G}\right)$
is a quadratic Lie algebra and
\textbf{observe that $B$ is not the restriction of $B_{\G}$ to $\g \times \g$,} 
as $V$ is isotropic. 
\end{proof}

\medskip
Within the context of 
a quadratic Hom-Lie algebra $(\g, \mu, T, B)$
with $T$ a centroid,
we have the following relationship between \textbf{Prop.\,\ref{proposicion cero}} 
and \textbf{Thm.\,\ref{teorema nuevo 2}}.

\medskip
\begin{Prop}\label{Prop cociclos}{\sl
Let $(\g,[\,\cdot\,,\,\cdot\,],B)$ be the quadratic Lie algebra of 
\textbf{Prop.\,\ref{proposicion cero}} 
and let $(\G,[\,\cdot\,,\,\cdot\,]_{\G},B_{\G})$
be that of \textbf{Thm.\,\ref{teorema nuevo 2}}. Then,

\medskip
\textbf{(i)} $\G$ is a central extension of $\g$ by $V$

\medskip
\textbf{(ii)} 
For each $1 \leq i \leq r$, let $D_i=\mu(a_i,\,\cdot\,):\g \to \g$, 
where $\g=\Im(T) \oplus \a$ and $\a=\Span_{\F}\{a_1,\ldots,a_r\}$. 
If all the $D_i$'s are inner derivations of the Lie algebra $(\g,[\,\cdot\,,\,\cdot\,])$, 
then the Hom-Lie bracket $\mu$ satisfies the Jacobi identity.

\medskip
\textbf{(iii)} 
If $\mu$ does not satisfy the Jacobi identity, 
then $\Ker(T)$ is equal to the center $C(\g)$ of the Lie algebra $(\g,[\,\cdot\,,\,\cdot\,])$\/.}
\end{Prop}

\medskip
\begin{proof}

\textbf{(i)} 
Let $x$ and $y$ be in $\g$, and let $u$ and $v$ be in $V$.
By \textbf{Prop.\,\ref{proposicion cero}}, 
the Lie bracket $[\,\cdot\,,\,\cdot\,]$ in $\g$ 
is given by $[x,y]=T(\mu(x,y))$.
Now use \eqref{definicion de k}, the $\mu$-invariance of $B$
and \eqref{corchete de extension central-1} to get,
\begin{equation}\label{corchete de ext central en terminos del corchete de Lie en g}
\begin{split}
[\,x+u\,,\,& y+v\,]_{\G}=k(\mu(x,y))\\
\,&=[x,y]+B\left(\mu(a_1,x),y\right)v_1+\ldots+B\left(\mu(a_r,x),y \right)v_r.
\end{split}
\end{equation}
For each $x\in \g$, consider the linear map $\ad_{\mu}(x)=\mu(x,\,\cdot\,):\g \to \g$. 
Since $T$ is a centroid for $(\g,\mu)$, it follows from 
\eqref{primera Hom-Lie Jacobi identity} that,
$$
\ad_{\mu}(x)\left([y,z]\right)=[\ad_{\mu}(x)(y),z]+[y,\ad_{\mu}(x)(z)],\ \ \text{ for all }z \in \g;
$$
that is, $\ad_{\mu}(x)$ is a derivation of the Lie algebra 
$(\g,[\,\cdot\,,\,\cdot\,])$. Since $B$ is invariant under $\mu$, it follows that
$\ad_{\mu}(x)$ is actually an anti-self-adjoint derivation with respect to $B$; {\it ie\/,}
$B\left(\ad_{\mu}(x)(y),z\right)=-B\left(y,\ad_{\mu}(z)\right)$.
Now, let $D_i=\ad_{\mu}(a_i)$, with $a_i$ running through 
the basis of $\a$ given in \textbf{Thm.\,\ref{teorema nuevo 2}}.
From \eqref{corchete de ext central en terminos del corchete de Lie en g}, 
the map $\theta:\g \times \g \to V$ defined by, 
$\theta(x,y)=B\left(D_1(x),y\right)v_1+\cdots+B\left(D_r(x),y\right)v_r$,
gives rise to a skew-symmetric bilinear map satisfying,
\begin{equation}\label{expresion para corchete de extension central 2}
\begin{split}
\,[x+u,y+v]_{\G}&=[x,y]+B\left(D_1(x),y\right)v_1+\cdots+B\left(D_r(x),y\right)v_r\\
\,&=[x,y]+\theta(x,y),\ \text{ for all }x,y \in \g \text{ and }u,v \in V.
\end{split}
\end{equation}
That is, $\theta$ is actually
the  2-cocycle of the Lie algebra $(\g,[\,\cdot\,,\,\cdot\,])$, with values in $V$
that yields $[\,\cdot\,,\,\cdot\,]_{\G}$,
thus proving the statement.

\medskip
\textbf{(ii)}
Since $T$ is a centroid, \eqref{primera Hom-Lie Jacobi identity} implies that,
\begin{equation}\label{MX}
\mu(T(x),\mu(y,z))+\mu(y,\mu(z,T(x)))+\mu(z,\mu(T(x),y))=0,
\end{equation}
for all $x$, $y$ and $z$ in $\g$. That is, the Jacobi identity holds true for 
$\mu$ if at least one of the elements involved in the cyclic sum belongs to $\Im(T)$. 
As $\g=\Im(T) \oplus \a$, it remains to verify whether or not the Jacobi identity also 
holds true for the elements in $\a$.
Observe that if
each $D_i=\mu(a_i,\,\cdot\,)$ is an inner derivation of the Lie algebra 
$(\g,[\,\cdot\,,\,\cdot\,])$ then, for each $1\le i\le r$, there exists 
an element $b_i$ in $\g$ such that $D_i(x)=[b_i,x]$, 
which is equivalent to $\mu(a_i,x)=\mu(T(b_i),x)$ for all $x$ in $\g$.  
Let $a$ and $a^{\prime}$ be in $\a$. Since $T$ is a centroid, we get,
\begin{equation}\label{casos}
{\aligned
\mu(a_i,\mu(a,a^{\prime})) & =\mu(T(b_i),\mu(a,a^{\prime})),\\
\mu(a,\mu(a^{\prime},a_i)) & =\mu(a,\mu(a^{\prime},T(b_i))),\quad\ \text{and,}\\
\mu(a^{\prime},\mu(a_i,a)) & =\mu(a^{\prime},\mu(T(b_i),a)),\quad\ \text{for all}\ 1 \le i \le r.
\endaligned}
\end{equation}
We saw before if at least one of the elements involved in the cyclic sum lies 
in $\Im(T)$, then the Jacobi identity holds true for $\mu$. 
Since $\a$ is generated by $a_1,\ldots ,a_r$, 
\eqref{MX} and \eqref{casos} imply that,
$$
\mu(a,\mu(a^{\prime},a^{\prime \prime}))+\mu(a^{\prime},\mu(a^{\prime \prime},a)+\mu(a^{\prime \prime},\mu(a,a^{\prime}))=0,\,\text{ for all }a,a^{\prime},a^{\prime \prime} \in \a.
$$
Therefore, $\mu$ satisfies the Jacobi identity.

\medskip 
\textbf{(iii)} 
It follows from \textbf{Thm.\,\ref{teorema nuevo 2}} that $V$ 
is an isotropic ideal of the quadratic Lie algebra
$(\G,[\,\cdot\,,\,\cdot\,]_{\G},B_{\G})$ which, by construction, 
is a central extension of $(\g,[\cdot\,,\cdot\,])$ by $V$ 
with cocycle $\theta(x,y)=B(D_1(x),y)v_1+\ldots+B(D_r(x),y)v_r$. 
By hypothesis, if $\mu$ does not satisfy the Jacobi identity, 
it follows from \textbf{(ii)} that at least one of the derivations $D_i$ is not inner.

\medskip
It was proved in \cite{GSV}, (see \textbf{Lemma\,2.2} and \textbf{Prop. 2.5.(ii)} there),
that under these conditions, for each $i$ there exists an $x_i$ in $\g$ 
such that, $D^{\prime}_i=D_i+[x_i,\,\cdot\,]$, and 
$\displaystyle{\cap_{i=1}^r\Ker(D^{\prime}_i)}=\{0\}$. 
Let $\a^{\prime}=\operatorname{Span}\{a_1+T(x_1),\ldots,a_r+T(x_r)\}$,
then $\g=\a \oplus \Im(T)=\a^{\prime} \oplus \Im(T)$
and $D^{\prime}_i=\mu(a_i,\,\cdot\,)+\mu(T(x_i),\,\cdot\,)
=D_i+[x_i,\,\cdot\,]$, for all $1\le i\le r$. 
Thus, the results obtained so far are valid for both 
$\a$ and $\a^{\prime}$. 
In addition, if the 2-cocylce $\theta^{\prime}$ arising from 
$D^{\prime}_1,\ldots,D^{\prime}_r$ differs from $\theta$ 
by a coboundary, then the central extensions in $\G$ 
associated to $\theta$ and $\theta^{\prime}$, are isomorphic. 
Therefore we might as well assume that $\displaystyle{\cap_{i=1}^r\Ker(D_i)}=\{0\}$.
Then, by \eqref{expresion para corchete de extension central 2}, 
the center of $(\G,[\,\cdot\,,\,\cdot\,]_{\G})$ is,
 $$
 C(\G)=C(\g) \cap \left\{\left(\cap_{i=1}^r\Ker(D_i)\right) \oplus V\right\}.
 $$
Whence, $C(\G)=V$ and $V^{\perp}=[\G,\G]_{\G}$.
By \textbf{\emph{Claim A4}}, $k$ satisfies $\Im(k)=V^{\perp}$ 
and $k\vert_{\g}=T$ (see \eqref{definicion de k}). 
If $\pi:\G\to\g$ is the canonical projection of the central extension,
we have, $\pi\left([\G,\G]_{\G}\right)=[\g,\g]$. It then follows that,
$$
\Im(T)=\pi(\Im(k))=\pi(V^{\perp})=\pi\left([\G,\G]_{\G}\right)=[\g,\g].
$$
Finally, since $T$ is $B$-self-adjoint, it follows that $\Ker(T)=C(\g)$.
\end{proof}

\medskip
The next result shows how to extend the
Hom-Lie algebra $(\g,\mu,T)$ to a Hom-Lie algebra
$(\G,\mu_{\G},L)$.

\medskip
\begin{Prop}\label{teorema nuevo 3}{\sl
Let $(\g,\mu,T,B)$ be a quadratic Hom-Lie algebra
with $T$ a centroid
and let $(\G,[\,\cdot\,,\,\cdot\,]_{\G}, B_{\G})$
be the Lie algebra in
\textbf{Thm.\,\ref{teorema nuevo 2}}.
Let $k:\g \to \G$ be the map defined in \eqref{definicion de k}\/.
Then, the vector space $\G=\g \oplus V$
has a Hom-Lie algebra structure
$\left(\G,\mu_{\G},L\right)$, where $\mu_{\G}(x+u,y+v)=\mu(x,y)$, with,
\begin{equation}\label{restriccion k}
L(x+u)=k(x),\,\,\text{ for all }x,y \in \g,\,\text{ and }u,v \in V,
\end{equation}
\begin{equation}\label{dieciseis}
\,\text{satisfying,}\ \ 
B_{\G}(L(x+u),y+v)=B_{\G}(x+u,L(y+v))=B(x,y),
\end{equation}
for all $x,y$ in $\g$ and $u,v$ in $V$. Moreover,
the projection map
$\pi:\G \to \g$ of the central extension,
also yields a Hom-Lie algebra epimorphism 
$(\G,\mu_{\G},L)\twoheadrightarrow(\g,\mu,T)$\/.}
\end{Prop} 

\medskip
\begin{proof}
Let $x,y,z$ be in $\g$ and let $u,v,w$ be in $V$. 
From the definition of $\mu_{\G}$ we deduce that $\mu_{\G}(\G,V)=\{0\}$. Then,
$$
\mu_{\G}\left(L(x+u),\mu_{\G}(y+v,z+w)\right)=\mu_{\G}\left(k(x),\mu(y,z)\right)=\mu(T(x),\mu(y,z)),
$$
and this implies that the triple $\left(\G,\mu_{\G},L\right)$ is a Hom-Lie algebra.
Now \eqref{restriccion k} says that
$L\vert_{\g}=k$ and $V \subset \Ker L$. 
Then, from
\textbf{Claim A5} and \eqref{relacion entre las metricas con k}
in the proof of \textbf{Thm.\,\ref{teorema nuevo 2}} we have,
$B(x,y)\!=\!B_{\G}(k(x),y+v)\!=\!B_{\G}(L(x),y+v)\!=\!B_{\G}(L(x+u),y+v)$. 
Now \eqref{dieciseis} follows
using the same arguments and the symmetry of $B$.

\medskip
Recall that 
\textbf{a morphism
$(\g_1,\mu_1,T_1)\to(\g_2,\mu_2,T_2)$
between two Hom-Lie algebras}
is a linear map $\psi:\g_1 \to \g_2$
such that $\psi(\mu_1(x,y))=\mu_2(\psi(x),\psi(y))$, for all $x,y$ in $\g_1$, 
and $\psi \circ T_1=T_2 \circ \psi$.
Now, $\pi(x+v)=x$ for any $x$ in $\g$ and $v$ in $V$.
Since, $\mu_{\G}(x+u,y+v)=\mu(x,y)$ already lies in $\g$, it follows that
$\pi(\mu_{\G}(x+u,y+v))=\mu\left(\pi (x+u),\pi(y+v)\right)$ 
and it is also true that $\pi\circ L(x+u)=T \circ \pi(x+u)$, 
for all $x,y$ in $\g$ and $u,v$ in $V$. 
Therefore, $\pi:(\G,\mu_{\G},L) \to (\g,\mu,T)$ is an
epimorphism of Hom-Lie algebras. 
Since $\mu_{\G}(\G,V)=\{0\}$ and $L(V)=\{0\}$, 
we may say that
$(\G,\mu_{\G},L)$ \textbf{is a central extension of}
$(\g,\mu,T)$ by $V$ as Hom-Lie algebras.
\end{proof}

\medskip
\begin{Remark}\label{Remark a twist map}{\rm
Observe that $\Ker(L)=V$
and that $\Im(L)=\Im(k)=V^\perp$,
which follows from \eqref{restriccion k}, \eqref{dieciseis} and \textbf{Claim A4}.}
\end{Remark}

\medskip
We now know from \textbf{Thm.\,\ref{teorema nuevo 2}} that
$B_{\G}:\G \times \G \to \F$ as defined in
\eqref{definicion de metrica en extension central}
is invariant for the Lie bracket $[\,\cdot\,,\,\cdot\,]_{\G}$, and 
from \eqref{dieciseis} in \textbf{Prop.\,\ref{teorema nuevo 3}}
we also know that
$L$ as defined in \eqref{restriccion k} is $B_{\G}$-self-adjoint.
It is now natural to ask whether or not ---or under what conditions---
$B_{\G}$ is $\mu_{\G}$-invariant and to also ask
whether or not $L$ 
lies in the centroid of
the Hom-Lie product $\mu_{\G}$.
These questions are settled down in the following:

\medskip
\begin{Cor}\label{corolario cociclos}{\sl
Under the hypotheses of \textbf{Prop.\,\ref{Prop cociclos}} 
and \textbf{Prop.\,\ref{teorema nuevo 3}},

\medskip
\textbf{(i)} If $B_{\G}$ is invariant under $\mu_{\G}$ then the 2-cocycle $\theta$ is zero. 

\medskip
\textbf{(ii)} $L$  
is a centroid of $\mu_{\G}$
if and only if $\theta$ is zero.}
\end{Cor}

\medskip
\begin{proof}
We shall adhere ourselves to the following notation. 
For a given non-degenerate symmetric bilinear form 
$B:\g\times\g\to\mathbb F$ on a given $\F$-space $\g$,
let $B^{\flat}:\g\to\g^\ast$ be the linear map defined by
$B^{\flat}(x)\,(y)=B(x,y)$, for all $x$ and $y$ in $\g$.
Since $\dim\g$ is assumed to be finite, $B^{\flat}$
is an isomorphism and so is its inverse map
$B^\sharp:\g^*\to\g$, characterized by the property,
$B\left(B^\sharp(\xi),x\right)=\xi(x)$, for any $\xi\in\g^\ast$
and any $x\in\g$. In what follows we shall use the already defined
non-degenerate symmetric bilinear forms $B$ and $B_{\G}$
on $\g$ and $\G$, respectively.

\medskip
\textbf{(i)} Now, let $\iota_{\g}:\g \hookrightarrow \G$ 
be the inclusion map and define,
\begin{equation}\label{definicion h}
h=B^{\sharp}\circ\,\iota_{\g}^*\,\circ\,B_{\G}^{\flat} :\G\to\g,
\end{equation}
where the dual map $S^\ast:V^\ast\to U^\ast$ of any linear map
$S:U\to V$ is defined by $S^\ast(\phi)=\phi\circ S$, for any $\phi\in V^\ast$.
It follows from \textbf{Claim A5} and the non-degeneracy of $B$ and $B_{\G}$ that:

\medskip
\begin{itemize}
\item[\textbf{(Q1)}] $h \circ k=\operatorname{Id}_{{\g}}$.

\medskip
\item[\textbf{(Q2)}] $B_{\G}(x+v,y)=B(h(x+v),y)$ for all $x,y \in \g$ and $v \in V$.

\medskip
\item[\textbf{(Q3)}] $\a=\Ker(h)={\g}^{\perp}$.

\medskip
\item[\textbf{(Q4)}] $\G=\Ker(h) \oplus \Im(k)$.
\end{itemize}

\medskip
We know from \eqref{corchete de extension central-1}
that $k(\mu(x,y))=[x,y]_{\G}$. Then \textbf{(Q1)} implies that,
\begin{equation}\label{mu y h}
\mu(x,y)=h([x,y]_{\G}),\text{ for all }x,y \in \g.
\end{equation}
On the other hand, from \textbf{(Q2)} one obtains,
\begin{equation}\label{B y h}
B_{\G}(h(x+u),y+v)=B_{\G}(x+u,h(y+v)),
\end{equation}
for all $x,y$ in $\g$ and all $u,v$ in $V$.
We can now prove the first part of \textbf{(i)}; that is,
if $B_{\G}$ is $\mu_{\G}$-invariant, then the 2-cocycle 
$\theta$ in \eqref{expresion para corchete de extension central 2}
has to be zero.
Indeed, let $x$ be in $\g$ and $y,z$ be in $\G$. 
The $\mu_{\G}$-invariance of $B_{\G}$, together with 
\eqref{mu y h}-\eqref{B y h}, the fact that
$\mu_{\G}\vert_{\g \times \g}=\mu$ (\textbf{Prop.\,\ref{teorema nuevo 3}}) 
and the invariance of $B_{\G}$ under $[\,\cdot\,,\,\cdot\,]_{\G}$,
allow us to write,
$$
B_{\G}(z,\mu_{\G}(x,y))\!=
\!\!B_{\G}(\mu_{\G}(z,x),y)\!=
\!\!B_{\G}\left(h\left([z,x]_{\G}\right),y\right)\!=
\!\!B_{\G}\left(z,[x,h(y)]_{\G} \right).
$$
The non-degeneracy of $B_{\G}$ and 
\eqref{expresion para corchete de extension central 2}
then imply that,
\begin{equation}\label{trece}
h\left([x,y]_{\G}\right)=\mu_{\G}(x,y)=[x,h(y)]_{\G}=[x,h(y)]+\theta(x,h(y))
\end{equation}
Since $\Im(h)=\g \subset \G$, equations 
\eqref{expresion para corchete de extension central 2} and \eqref{trece}
show that $\theta(x,h(y))$ lies in $\g \cap V=\{0\}$. 
Being $h$ is surjective and $x$ and $y$ arbitrary,
one concludes that $\theta=0$. 

\medskip
\textbf{(ii)} We now prove that,
if $L$ is a centroid of $\mu_{\G}$,
then $\theta$ is zero. Take $x,y$ in $\g$ and $u,v$ in $V$. 
From \textbf{Prop.\,\ref{teorema nuevo 3}.(i)} we have, 
$$
L(\mu_{\G}(x+u,y+v))-\mu_{\G}(L(x+u),y+v)=k(\mu(x,y))-\mu_{\G}(k(x),y).
$$
Now, using 
\eqref{corchete de extension central-1}, 
\eqref{expresion para corchete de extension central 2} and 
\eqref{mu y h}, together with the fact that 
$\mu_{\G}\vert_{\g \times \g}=\mu$, we obtain,
$$
\aligned
k(\mu(x,y))-\mu_{\G}(k(x),y)&=[x,y]_{\G}-h([k(x),y]_{\G})\!=\![x,y]_{\G}\!-\!h(k([x,y]))\\
\,&=[x,y]_{\G}-[x,y]=\theta(x,y),\\
\endaligned
$$
where use has been made of
$k([x,y])=[k(x),y]_{\G}$ and $h\circ k=\Id_{\g}$
(see \textbf{Claim A1} and \textbf{(Q1)}). 
Thus, $L(\mu_{\G}(x+u,y+v))-\mu_{\G}(L(x+u),y+v)=\theta(x,y)$,
which shows that 
$L$ is a centroid
precisely when $\theta=0$.
\end{proof}

\medskip
\textbf{Remark.} 
Our next result provides a method to generate quadratic
Hom-Lie algebras on an underlying vector space of the
form $\g=\h\oplus\h^\ast$, when a Hom-Lie algebra
$(\h,\nu,S)$ has been given with a centroid
$S$. In fact, we shall close this section by applying
such a method to a $6$-dimensional $\h$, 
which is in turn interesting in itself, as its
Hom-Lie product $\nu:\h\times\h\to\h$ {\it is not a Lie bracket\/.}

\medskip
\begin{Prop}\label{construccion}{\sl
Let $(\h,\nu,S)$ be a Hom-Lie algebra with $S$ 
a centroid for $\nu$.
Let $\nu^{\ast}:\h \to \gl(\h^{\ast})$ be the linear map for which 
$\nu^{\ast}(x)(\alpha)\in\h^\ast$, for any $x\in\h$ and any $\alpha\in\h^\ast$,
is the linear map $y\mapsto -\alpha(\nu(x,y))$, for any $y$ in $\h$.
Let $\g=\h\oplus\h^\ast$ and define $\mu:\g\times\g\to\g$ by,
$$
\mu(x+\alpha,y+\beta)=\nu(x,y)+\nu^{\ast}(x)(\beta)-\nu^{\ast}(y)(\alpha).
$$
Let $T:\g\to\g$ be given by, $T(x+\alpha)=S(x)+S^{\ast}(\alpha)$, 
and let $B:\g\times\g\to{\mathbb F}$ be defined by
$B(x+\alpha,y+\beta)=\alpha(y)+\beta(x)$. 
Then, $(\g,\mu,T,B)$ is a quadratic Hom-Lie algebra,
having $T$ in its centroid\/.}
\end{Prop}

\medskip
\begin{proof}
We shall prove first 
that $T$ is a centroid.
For $x$ and $y$ in $\h$ and $\alpha$ in $\h^{\ast}$, we have,
on the one hand, $T(\mu(x,\alpha))(y)=-\alpha(\nu(x,S(y)))$.
On the other hand, it is clear that,
$$
\mu(T(x),\alpha)(y)=-\alpha(\nu(S(x),y)),\ \,\text{and}\ \,
\mu(x,T(\alpha))(y)=-\alpha(S(\nu(x,y))).
$$
Whence, $T(\mu(x,\alpha))=\mu(T(x),\alpha)=\mu(x,T(\alpha))$.

\medskip
We shall now prove that the Hom-Lie Jacobi identity 
\eqref{primera Hom-Lie Jacobi identity} holds true for $\mu$ and $T$. 
Take $x,y,z$ in $\h$ and $\alpha$ in $\h^{\ast}$. 
Using the fact that $T$ is a centroid,
one obtains,
$$
\aligned
&\mu(T(x),\mu(y,\alpha))(z)+\mu(T(y),\mu(\alpha,x))(z)+\mu(T(\alpha),\mu(x,y))(z)\\
&=\alpha\left((\nu(S(y),\nu(x,z))+ (\nu(S(x),\nu(y,z))+ (\nu(S(z),\nu(y,x))\right)=0.
\endaligned
$$
Thus, $\mu(T(x),\mu(y,\alpha))+\mu(T(y),\mu(\alpha,x))+\mu(T(\alpha),\mu(x,y))=0$.
The fact that $T$ is self-adjoint for the given $B$
follows from $B(S(x),\alpha)=\alpha(S(x))=S^{\ast}(\alpha)(x)=B(x,S^{\ast}(\alpha))$,
 since $T(x)=S(x)$ and $T(\alpha)=S^\ast(\alpha)$.
\end{proof}

\medskip
\textbf{An example with a $6$-dimensional Hom-Lie algebra 
$(\h,\nu,S)$ with $\nu$-equivariant $S$
and $\nu$ not a Lie bracket.}

\medskip
Let $\h=\operatorname{Span}_{\F}\{x_1,x_2,x_3,y_1,y_2,y_3\}$ 
with the skew-symmetric bilinear product $\nu:\h \times \h \to \h$ defined by
$\nu(x_1,x_2)\!=\!\nu(y_3,x_1)\!=\!\nu(x_3,y_1)\!=\!y_2$, 
$\nu(x_3,x_1)\!=\!y_1$, and all other products on the basis being equal to zero. 
Let $S:\h \to \h$ be the linear map defined by,
$$
S(x_1)=y_1,\ S(x_2)=y_2,\ S(x_3)=y_3,\, S(y_1)=y_2,\ S(y_2)=S(y_3)=0.
$$
Observe that both, $\nu(\h,\h)$ and $\Im(S)$, lie in 
$\operatorname{Span}_{\F}\{y_1,y_2,y_3\}$, and 
$\nu(y_j,y_k)=0$, for all $1\le j,k\le 3$. 
It is easy to check that $(\h,\nu,S)$ is a 6-dimensional Hom-Lie algebra,
and that $S$ is 
in the centroid.
Also observe that $S$ is nilpotent with $\dim\Ker(S)=3$. In addition,
$$
\nu(x_1,\nu(x_2,x_3))+\nu(x_2,\nu(x_3,x_1))+\nu(x_3,\nu(x_1,x_2))=-y_2 \neq 0,
$$
which proves that $\nu$ is not a Lie bracket. 
Applying \textbf{Prop.\,\ref{construccion}}, we get a 12-dimensional 
quadratic Hom-Lie algebra $(\h \oplus \h^{\ast},\mu,T,B)$, 
where $T$ is 
in the centroid
and $\dim\Ker(T)=6$. 

\medskip
Let $\h^{\ast}=\operatorname{Span}_{\F}\{\alpha_1,\alpha_2,\alpha_3,\beta_1,\beta_2,\beta_3\}$, 
where $\alpha_{i}(x_j)=\beta_{i}(y_j)=\delta_{ij}$, for all $1 \leq i,j \leq 3$. 
Write $\g=\h \oplus \h^{\ast}$. 
From the prescription of \textbf{Prop.\,\ref{construccion}}
we obtain the Hom-Lie product $\mu$ on $\g$ given by,
$$
\aligned
& \mu(x_1,x_2)=\mu(y_3,x_1)=\mu(x_3,y_1)=y_2,\,\,\,\mu(x_2,x_3)=y_3,\,\,\,\mu(x_3,x_1)=y_1,\\
& \mu(x_1,\beta_2)=-\alpha_2+\beta_3,\,\,\,\,\,\,\,\mu(x_2,\beta_3)=\mu(\beta_1,x_1)=\mu(y_1,\beta_2)=-\alpha_3,\\
& \mu(x_2,\beta_2)=\mu(\beta_1,x_3)=\mu(\beta_2,y_3)=\alpha_1,\,\,\,\,\mu(x_3,\beta_2)=\beta_1,\,\,\,\mu(x_3,\beta_3)=\alpha_2,
\endaligned
$$
and the twist map $T:\g \to \g$ given by,
$$
\aligned
& T(x_1)=y_1,\,\,\,T(x_2)=y_2,\,\,\,T(x_3)=y_3,\,\,\,T(y_1)=y_2,\,\,\,T(y_2)=T(y_3)=0,\\
& T(\beta_1)=\alpha_1,\,\,\,T(\beta_2)=\alpha_2+\beta_1,\,\,\,T(\beta_3)=\alpha_3.
\endaligned
$$
It is easy to determine the Lie algebra bracket
$[x,y]=T(\mu(x,y))$ for all $x,y$ in $\g$. In fact,  
the only non-trivial brackets are:
$[x_3,x_1]=y_2$, $[x_1,\beta_2]=\alpha_3$ and $[x_3,\beta_2]=\alpha_1$. 
Thus $(\g,[\,\cdot\,,\,\cdot\,])$ is 2-step nilpotent Lie algebra
with 
$C(\g)=\operatorname{Span}
\{x_2,y_1,y_2,y_3,\alpha_1,\alpha_2,\alpha_3,\beta_1,\beta_3\}$, 
and $[\g,\g]=\operatorname{Span}\{y_2,\alpha_1,\alpha_3\}$.

\medskip
\section{Hom-Lie algebras and algebras having no non-trivial two-sided ideals}

\medskip
The main result of this section is that we can 
recover the {\it Hom-Lie algebra structure\/} in $\G=\g \oplus V$
given in \textbf{Prop.\,\ref{teorema nuevo 3}}, from an algebra $\mathcal{A}$
having no non-trivial two-sided ideals. 
In order to accomplish this goal, we shall prove first 
the existence of a bilinear product $(x,y)\mapsto xy$ in $\G$ that 
realizes  $\mu_{\G}$ in the form $\mu_{\G}(x,y)=xy-yx$.
This product will help us to construct the algebra $\mathcal{A}$.

\medskip
\begin{Prop}\label{proposicion conexion}{\sl
Let $(\g,\mu,T,B)$ be a quadratic Hom-Lie algebra
with $T$ in the centroid and $\dim\Ker(T)>0$. 
Choose a vector space $V$ with $\dim V=\dim\Ker(T)$
to produce the
quadratic Lie algebra $\left(\G=\g \oplus V,[\,\cdot\,,\,\cdot\,]_{\G},B_{\G}\right)$ 
of \textbf{Thm.\,\ref{teorema nuevo 2}}, and let $(\G,\mu_{\G},L)$
be the Hom-Lie algebra $(\G,\mu_{\G},L)$ of \textbf{Prop.\,\ref{teorema nuevo 3}}. 
There is a bilinear product $\G\times\G\to\G$, 
denoted by $(x,y)\mapsto xy$ (for all $x$ and $y$ in $\G$),
such that,

\medskip
\textbf{(i)} $\mu_{\G}(x,y)=xy-yx$.

\medskip
\textbf{(ii)} $2\,xy=\mu_{\G}(x,y)-[h(x),y]_{\G}+[x,h(y)]_{\G}$, 
with $h:\G \to \g$ as in \eqref{definicion h},

\medskip
\textbf{(iii)} $x^{2}=0$ for all $x$ in  $\Ker(h) \cup \Im(L)$, 
and $\G=\Ker(h) \oplus \Im(L)$.}
\end{Prop}

\medskip
\begin{proof}
Define the bilinear map $(x,y)\mapsto xy\in \G$ by means of,
\begin{equation}\label{0}
2B_{\G}(xy,z)\!=\!B_{\G}\left(\mu_{\G}(x,y),z\right)\!
+\!B_{\G}\left(\mu_{\G}(z,x),y\right)\!+\!B_{\G}\left(\mu_{\G}(z,y),z\right),
\end{equation}
for all $x,y$ and $z$ in $\G$. If $B_{\G}$ is invariant under 
$\mu_{\G}$, then $2\,x\,y=\mu_{\G}(x,y)$. 
By \textbf{Cor.\,\ref{corolario cociclos}.(i)}, this is the case 
only if the 2-cocycle $\theta$ is zero.
Now observe that \textbf{(i)} follows from \eqref{0}. 
To prove 
\textbf{(ii)}
take $x,y$ and $z$ in $\G$. 
The properties of $h$ given in \eqref{mu y h} and \eqref{B y h}, 
and the invariance of $B_{\G}$ under $[\,\cdot\,,\,\cdot\,]_{\G}$, 
lead to the following chain of equalities:
\begin{align}
2 \,B_{\G}(xy,z)=&B_{\G}(\mu_{\G}(x,y),z)+B_{\G}(h([z,x]_{\G}),y)+B_{\G}\left(h([z,y]_{\G}\right),x)\nonumber\\
=& B_{\G}(\mu_{\G}(x,y),z)+B_{\G}\left(z,[x,h(y)]_{\G}\right)+B_{\G}\left(z,[y,h(x)]_{\G}\right)\nonumber\\
=& B_{\G}\left(\mu_{\G}(x,y)+[x,h(y)]_{\G}-[h(x),y]_{\G},z\right),\nonumber
\end{align}
Therefore,
\begin{equation}\label{conexion explicita}
 2xy\!=\!\mu_{\G}(x,y)\!+\![x,h(y)]_{\G}\!-\![h(x),y]_{\G},\,\text{ for all }x,y \in \G.
\end{equation}
Finally,  to prove \textbf{(iii)}, 
we shall first show the following:

\medskip
\textbf{Claim B1.} {\it $2L(x+u)L(y+v)\!\!=\!\!T([x,y])$, for all $x,\!y$ in $\g$ and $u,\!v$ in $\!V\!$\/.}

\medskip
To prove this claim we shall use the fact that $L\vert_{\g}=k$ (see \eqref{restriccion k}),
together with $h \circ k=\Id_{\g}$ (see \textbf{(Q1)}), \textbf{Claim A1}, and 
the fact that $h([x,y]_{\G})=\mu(x,y)$ (see \eqref{mu y h}). 
Thus, by \eqref{conexion explicita} we get,
\begin{equation}\label{T primas}
{\aligned
2\,L(x+u) L(y+v) & =2k(x)k(y)\\
&=h\left([k(x),k(y)]_{\G}\right)+[k(x),y]_{\G}-[x,k(y)]_{\G}\\
&=h\left([k(x),T(y)]_{\G}\right)=h\left(k(\mu(T(x),T(y))\right)\\
&=\mu(T(x),T(y)).
\endaligned}
\end{equation}
Since $T$ is a centroid and $\mu(T(x),y)=[x,y]$, 
it follows from the right hand side of \eqref{T primas} that, $\mu(T(x),T(y))=T([x,y])$,
thus proving \textbf{Claim B1.} 
In particular, $L(x+u)^2=0$.

\medskip
To conclude the proof, we now only have to prove the following:

\medskip
\textbf{Claim B2.}
{\it $\,2\,a a^{\prime}=\mu_{\G}(a,a^{\prime})$, for all $a,a^{\prime}$ in $\a$\/.}

\medskip
Let $a$ and $a^{\prime}$ be in $\a=\Ker(h)$ (see \textbf{(Q3)}).
From \eqref{conexion explicita} we get, 
$2a a^{\prime}=\mu_{\G}(a,a^{\prime})$ and $a^2=0$. 
Therefore, $x^2=0$ for all $x$ in $\Im(L) \cup \Ker(h)$.
\end{proof}

\medskip
\begin{Theorem}\label{teorema}{\sl
Let $(\g,\mu,T,B)$ be a quadratic Hom-Lie algebra such that $T$ 
is a centroid and $\dim(\Ker(T))=r>0$. Let $(\G,[\,\cdot\,,\,\cdot\,]_{\G},B_{\G})$
be the quadratic Lie algebra of \textbf{Thm.\,\ref{teorema nuevo 2}}
and let $(x,y) \mapsto xy$ be the product in $\G$ of \textbf{Prop.\,\ref{proposicion conexion}}.
\medskip
\textbf{(i)} Let $\mathcal{A}=
\F \times \G$
and define the following multiplication in $\mathcal{A}$:
$$
(\xi,x)(\eta,y)=(\,\xi\, \eta+B_{\G}(x,y)\,,\,\xi\, y+\eta\, x+x\,y\,),
$$
for all $\xi$ and $\eta$ in $\F$ and all $x$ and $y$ in $\G$.
Then $\mathcal{A}$ is a non-abelian, and in general, a non-associative
algebra with unit element $1_{\mathcal{A}}=(1,0)$.

\medskip
\textbf{(ii)}
If $\mu$ is not a Lie product and $(\g,[\,\cdot\,,\,\cdot\,]=T \circ \mu)$
is not Abelian, then $\mathcal{A}$ has no non-trivial two-sided ideals. 
Furthermore, if $\left(\g,[\,\cdot\,,\,\cdot\,]\right)$ is nilpotent then $\mathcal{A}$ is simple.

\medskip
\textbf{(iii)} Let $\mathcal{A}^{\prime}=\mathcal{A}/\mathbb F\,1_{\mathcal{A}}$ 
and $\mu^{\prime}:\mathcal{A}^{\prime} \times \mathcal{A}^{\prime} 
\to \mathcal{A}^{\prime}$ be defined by,
$$
\mu^{\prime}\left((\xi,x)+\F 1_{\mathcal{A}},(\eta,y)+\F 1_{\mathcal{A}}\right)=
(\xi,x)(\eta,y)-(\eta,y)(\xi,x)+\F\, 1_{\mathcal{A}},
$$
for all $\xi$ and $\eta$ in $\F$ and all $x$ and $y$ in $\G$. 
Let $T^{\prime}:\mathcal{A}^{\prime} \to \mathcal{A}^{\prime}$ 
be the linear map defined by 
$T^{\prime}((\xi,x)+\F 1_{\mathcal{A}}))=(\xi,L(x))+\F 1_{\mathcal{A}}$. 
Then, $\left(\mathcal{A}^{\prime},\mu^{\prime},T^{\prime}\right)$ 
is a Hom-Lie algebra isomorphic to 
$\left(\G,\mu_{\G},L \right)$ of \textbf{Prop.\,\ref{teorema nuevo 3}}.}
\end{Theorem}

\medskip
\begin{proof}
\textbf{(i)} 
We define equality, addition and multiplication by scalars
in $\mathcal{A}$ in the obvious manner.
Then $(1,0)$ is the unit element $1_{\mathcal{A}}$ of $\mathcal{A}$.
Observe that the conditions for the product in $\mathcal{A}$ 
to be associative are in general too restrictive; namely, the
following identities would have to be satisfied for any $x$, $y$ and $z$ in $\G$:
\textbf{(i)} $x\,(y\,z) = (x\,y)\,z$;
\textbf{(ii)} $B_{\G}(x\,y,z) = B_{\G}(x,y\,z)$; and
\textbf{(iii)} $B_{\G}(x,y)\,z = B_{\G}(z,x)\,y$.
Similarly, for $\mathcal{A}$ to be abelian, it would have to be true
that $\mu_{\G}(x,y)=\mu_{\G}(y,x)$ for all $x$ and $y$ in $\G$;
since $\mu_{\G}$ is assumed to be skew-symmetric, this leads to
 $\mu_{\G}\equiv 0$.

\medskip
\textbf{(ii)} We shall prove that any non-trivial two-sided ideal in $\mathcal{A}$
contains the unit element $1_{\mathcal{A}}$. 
We proceed by proving the following statements:

\medskip
\textbf{Claim C1.} {\it $x L(y)$ belongs to $\Im(L)$ for all $x,y$ in $\G$\/.}

\medskip
We shall prove first that $x L(y)$ belongs to $\Im \left(L\right)$ 
for $x$ and $y$ in $\g$. Using the same arguments we used in
the proof of \textbf{Claim B1}, we have,
\begin{equation}\label{catorce}
\begin{split}
2\,x\,L(y)=2x\,k(y)&=h([x,k(y)]_{\G})+[x,y]_{\G}-[h(x),k(y)]_{\G}\\
\,&=h(k(\mu(T(x),y)))+[x,y]_{\G}-k([h(x),y])\\
\,&=T(\mu(x,y))+[x,y]_{\G}-k([h(x),y]).
\end{split}
\end{equation}
From \textbf{(P2)} and \textbf{Remark \ref{Remark a twist map}} we know that 
$[x,y]=T(\mu(x,y))$ belongs to $\Im(T) \subset V^{\perp}=\Im(k)=\Im(L)$. 
It follows from 
\eqref{corchete de extension central-1} that 
$[x,y]_{\G}-k([h(x),y])=k(\mu(x,y))-k([h(x),y]_{\G})$ belongs to $\Im(k)=\Im(L)$. 
It then follows from \eqref{catorce} that $2\,x\,L(y)=[x,y]+[x,y]_{\G}-k([h(x),y])$ 
belongs to $\Im(L)$, and $x L(y)$ is an element of $\Im(L)$. 
We may now prove that $uL(y)$ belongs to $\Im(L)$, when $u$ is in 
$V$ and $y$ in $\g$. 
Since $V$ is isotropic, $u$ is an element of $V^{\perp}=\Im(k)=\Im(L)$. 
Thus, one may find $z$ in $\g$ and $v$ in $V$, such that $u=L(z+v)=L(z)$. 
By \eqref{T primas} we deduce that $2uL(y)=L(z)L(y)=T([z,y])$,
which belongs to $\Im(T) \subset V^{\perp}=\Im(k)=\Im(L)$ 
(see  \textbf{(P2)} and \textbf{Remark \ref{Remark a twist map}}). 
Therefore, $(x+u)L(y)$ belongs to $\Im(L)$. 

\medskip
Let us now consider a two-sided ideal $I \neq \{0\}$ of $\mathcal{A}$.

\medskip
\textbf{Claim C2.} {\it $I$ has an element of the form $(\xi,L(x))$, with $\xi \neq 0$\/.}

\medskip
Let $(\xi,x+u) \neq 0$ be in $I$, with $\xi$ in $\F$, $x$ in $\g$ and $u$ in $V$.
If $\xi=0$, then $x+u$ is non-zero and there must be an element $y^{\prime}$
in $\G$ such that $B_{\G}(x+u,y^{\prime})=1$. 
Then $(0,x+u)(0,y^{\prime})=(1,(x+u)y^{\prime})$ belongs to $I$. 
Thus, we may assume from the start that $\xi$ is non-zero. 
Let $y$ be in $\g$. 
Since $B(x,y)=B_{\G}(x+u,k(y))=B_{\G}\left(x+u,L(y)\right)$
(see \eqref{restriccion k} and \eqref{dieciseis}), we obtain,
\begin{equation}\label{aux12}
\begin{split}
(\xi,x+u)(0,L(y))&=
\left(B_{\G}(x+u,L(y)),\xi L(y)+xL(y)\right)\\
\,&=\left(B(x,y),\xi L(y)+xL(y)\right) \in I.
\end{split}
\end{equation}
Now from \textbf{Claim C1} we know that $\xi L(y)+x L(y)$ belongs to $\Im(L)$. 
If $x \neq 0$ we may  choose a $y$ such that $B(x,y) \neq 0$ and the claim follows
for this case.
If $x=0$, then from \eqref{aux12} we have that $(0,\xi L(y))$ 
belongs to $I$, for all $y$ in $\g$ and $\xi \neq 0$. 
Then, $(0,L(y))(0,L(z))$ lies in $I$, for all $y$ and $z$ in $\g$.
Let $y$ and $z$ be such that $B_{\G}(L(y),L(z)) \neq 0$. 
It then follows from
\eqref{restriccion k}, \eqref{dieciseis} and \eqref{T primas} that,
\begin{equation}\label{diecisiete}
2\,(0,L(y))(0,L(z))=(2\,B_{\G}(L(y),L(z)),T([y,z])) \in I.
\end{equation}
From \textbf{(P2)} and \textbf{Remark\,\ref{Remark a twist map}}
we know that $\Im(T) \subset V^{\perp}=\Im(k)=\Im(L)$ and by \eqref{diecisiete} 
the claim follows in the general case.

\medskip
\textbf{Claim C3.} 
{\it If $(\xi,L(x))$ lies in $I$ with $-\xi^2\!+\!B_{\G}(L(x),L(x))\!\neq \!0$, 
then $(1,0)$ belongs to $I$\/.}

\medskip
Indeed.
By \textbf{Prop.\,\ref{proposicion conexion}.(iii)}, $L(x)^2=L(x)L(x)=0$.
Then,
 $$
(\xi,L(x))(-\xi,L(x))=\left(-\xi^2+B_{\G}(L(x),L(x)),0\right).
$$ 
If this is non-zero, then $(1,0)$ lies in $I$. 
Thus, we shall now make the following assumption:

\medskip
$[*]$ \textbf{Assumption.} {\it $\xi^2=B_{\G}(L(x),L(x))$, for all $(\xi,L(x))$ in $I$\/.}

\medskip
\textbf{Claim C4.} 
{\it Assuming $[*]$, if $(\xi,L(x))$ lies in $I$, with $\xi\ne 0$ and  $x$ in $\g$,
then $x$ belongs to $\Ker(T)=C(\g)$\/.}

\medskip
Since $L(x)L(y)=-L(y)L(x)$ (see \textbf{Claim B1}), 
for any $\left(\eta,L(y)\right)$ in $\mathcal{A}$, with $y$ in $\G$, 
we have that,
\begin{equation}\label{quince}
\begin{split}
(\xi,L(x))(\eta,L(y)) & +(\eta,L(y)(\xi,L(x))
\\
&=
2\,(
\xi \eta + B_{\G}(L(x),L(y)),\xi L(y)+\eta L(x)),
\end{split}
\end{equation}
belongs to $I$.
Now, using $[*]$, with $\xi$ replaced by $\xi \eta+B_{\G}(L(x),L(y))$ 
and $L(x)$ replaced by $\xi L(y)+\eta L(x)$, \eqref{quince} implies that,
\begin{equation}\label{polarization}
B_{\G}(L(x),L(y))^2=\xi^2B_{\G}(L(x),L(y)),\quad\text{ for all }y \in \G.
\end{equation}
By polarization this identity in turn implies that for any $y$ and $z$ in $\G$,
\begin{equation}\label{linealizacion}
B_{\G}\left(L(x),L(y)\right)
B_{\G}\left(L(x),L(z)\right)
=\xi^2B_{\G}\left(L(y),L(z)\right).
\end{equation}
Since $L$ is $B_{\G}$-self-adjoint and $B_{\G}$ is non-degenerate,
\begin{equation}\label{p1}
B_{\G}\left(L(x),L(y)\right)\,
L\left(L(x)\right)
=\xi^2\,L\left(L(y)\right),
\quad \text{for any\ } y \in \G.
\end{equation}
Knowing that $L\vert_{\g}=k$ (see \eqref{restriccion k}) 
and $V=\Ker(L)$ (see \textbf{Remark\,\ref{Remark a twist map}}),
we have,
\begin{align}
\label{p2} \,& L(L(y))=L(k(y)),\qquad\text{and,}\\
\label{p2-2} \,& B_{\G}(L(x),L(y))=B_{\G}(k(x),k(y)),\,\,\text{ for all }y \in \g.
\end{align}
Now substitute \eqref{p2} and \eqref{p2-2} in \eqref{p1}, to get,
\begin{equation}\label{jh}
B(k(x),k(y))\,k(T(x))=\xi^2k(T(y)),\quad \text{for all\ }y \in \g.
\end{equation}
Since $h \circ k=\operatorname{Id}_{\g}$, one may 
apply $h$ in \eqref{jh} to obtain,
\begin{equation}\label{last1}
B_{\G}(k(x),k(y))T(x)=\xi^2T(y),\quad \text{for all\ }y \in \g.
\end{equation}
We substitute $y$ by $\mu(x,y)$ in \eqref{last1} and 
apply the fact that $k(\mu(x,y))=[x,y]_{\G}$ (see \eqref{corchete de extension central-1}), 
together with $[k(x),x]_{\G}=k([x,x])=0$ (see \textbf{Claim A1}) and $T(\mu(x,y))=[x,y]$. 
Therefore,
\begin{equation}\label{last2}
\xi^2[x,y]=B_{\G}(k(x),[x,y]_{\G})T(x)=0,
\end{equation}
Since $\xi \neq 0$, \eqref{last2} says that $x$ lies in $C(\g)=\Ker(T)$
(see \textbf{Prop.\,\ref{Prop cociclos}.(iii)}), 
thus proving \textbf{Claim C4}. 
Also observe that $L\vert_{\g}=k$ and $V=\Ker(L)$
imply that $L^2(x)=k(T(x))=0$. 
Using this fact in \eqref{p1}, we obtain $L^2=0$
and therefore, $L^2(y)=k(T(y))=0$ for all $y$ in $\g$. 
Since $k$ is injective, $T=0$. 
Also, the definition \eqref{definicion de k} of $k$ implies that $k(\g)$ 
is contained in $V$. Since $[\G,V]_{\G}=\{0\}$, it follows that 
$\{0\}=[k(\g),\g]_{\G}=k([\g,\g])$ (see \textbf{Claim A1}).
And again, $k$ injective implies that the Lie algebra 
$(\g,[\,\cdot\,,\,\cdot\,])$ is Abelian, which is a contradiction.
Therefore, $[*]$ must be false; that is, there is an element $(\xi,L(x))$ in $I$ 
such that
$-\xi^2+B_{\G}\left(L(x),L(x)\right) \ne 0$,
which, by \textbf{Claim C3} implies that $(1,0)$ lies $I$. 

\medskip
Finally, we shall prove that if 
$(\g,[\,\cdot\,,\,\cdot\,])$ is a non-Abelian nilpotent Lie algebra,
then $\mathcal{A}$ is simple.
Let $I \neq \{0\}$ be a right ideal in $\mathcal{A}$ and
assume that $[*]$ is true.
Let $(\xi,L(x))$ be in $I$.
Using the same argument as in \textbf{Claim C2}, we may further assume 
that $\xi$ is non-zero. Take $y$ in $\g$ and $v$ in $V$.
From \eqref{conexion explicita} and \eqref{T primas} we have,
\begin{equation}\label{right 1}
(\xi,L(x))(0,2L(y))=(2B_{\G}(L(x),L(y)),2\,\xi\,L(y)+T\left([x,y]\right)).
\end{equation}
Now apply the assumption $[*]$ in this equation to get,
\begin{equation}\label{right-1-2}
\begin{split}
4B_{\G}(L(x),L(y))^2&=\!4\xi^2 B_{\G}(L(y),L(y))\!+\!4\xi B_{\G}(L(y),T([x,y]))\\
\,&+B_{\G}(T([x,y]),T([x,y])).
\end{split}
\end{equation}
Since $B_{\G}(T(x),T(y))=B(T(x),y)$
(see \eqref{definicion de metrica en extension central}),
$L\vert_{\g}=k$ and $\Im(T) \subset V^{\perp}$, we obtain,
\begin{equation}\label{sustitucion}
\begin{split}
& B_{\G}(L(y),T([x,y]))=B(T(y),T([x,y]))=0,\quad\text{ and, }\\
& B_{\G}(T([x,y]),T([x,y])=B([x,y],T([x,y])).
\end{split}
\end{equation}
Now substitute \eqref{sustitucion} in \eqref{right-1-2} 
and use the fact that $B_{\G}(L(x),L(y))=B(T(x),y)$
(see \eqref{definicion de metrica en extension central}) to further obtain that,
\begin{equation}\label{right 2}
4\,B(T(x),y)^2=4\,\xi^2B(T(y),y)+B([x,y],[T(x),y]),\text{ for all }y \in \g.
\end{equation}
Then, applying \eqref{right-1-2} to $y+z \in\G$, we obtain,
\begin{equation}\label{linearizacion 2}
\aligned
4\,B(T(x),y)B(T(x),z)\!=4\,\xi^2B(T(y),z)+B([x,T(y)],[x,z]),
\endaligned
\end{equation}
for all $y$ and $z$ in $\g$. 
Since $B$ is an invariant metric
on $(\g,[\,\cdot\,,\,\cdot\,])$ and $T$ is $B$-self-adjoint,
it follows from \eqref{linearizacion 2} that,
\begin{equation}\label{linearizacion 3}
4\,B(x,T(y))T(x)=4\,\xi^2 T(y)-[x,[x,T(y)]],\,\,\text{ for all }y \in \g.
\end{equation}
Now write $\mu(x,y)$ instead of $y$ in \eqref{linearizacion 3}, 
and use the fact that $T\circ \mu=[\,\cdot\,,\,\cdot\,]$ 
together with the fact that $B$ is invariant under $[\,\cdot\,,\,\cdot\,]$ to 
conclude that,
\begin{equation}\label{right 31}
0=4\,\xi^2 [x,y]-[x,[x,[x,y]]],\quad\text{ for all }y \in \g.
\end{equation}
If $\ad(x)^2=0$, then $4\,\xi^2 [x,y]=0$ by \eqref{right 31}. 
As $\xi \neq 0$ and $y$ is arbitrary, this implies that 
$x$ belongs to $C(\g)=\Ker(T)$ (see \textbf{Prop.\,\ref{Prop cociclos}.(iii)}). 
Now by \eqref{linearizacion 2}, $B(T(y),z)=0$ for all $y$ and $z$, so that $T=0$, 
and $k([\g,\g])=k(\mu(T(\g),\g))=\{0\}$ (see \eqref{definicion de k}). 
Finally, since $k$ is injective, $[\g,\g]=\{0\}$. 
On the other hand, if $\ad(x)^{2} \neq \{0\}$, 
\eqref{right 31} implies that $4\,\xi^2 \neq 0$ is an eigenvalue of 
$\ad(x)^{2}$ associated to the eigenvector $\ad(x)(y)$. 
But $(\g,[\,\cdot\,,\,\cdot\,])$ is nilpotent. Therefore, $\ad(x)^2$ 
is a nilpotent linear map and $\ad(x)(y)=0$, for all $y$ in $\g$,
which implies that $x$ belongs to $C(\g)=\Ker(T)$. 
We may now apply the same argument we used in the previous case
to conclude that $[\g,\g]=\{0\}$.

\medskip
In summary, if $[*]$ holds true, then $[\g,\g]=\{0\}$. 
Therefore, if $[\g,\g]\neq \{0\}$, then $[*]$ must be false 
and $\mathcal{A}$ has no non-trivial right-ideals. 
The proof for left-ideals is similar.

\medskip
\textbf{(iii)} From the definition of $\mu^{\prime}$ and 
the definition of the product $(x,y) \mapsto xy$, it follows that,
$\mu^{\prime}\left((\xi,x)+\F 1_{\mathcal{A}},
(\eta,y)+\F 1_{\mathcal{A}}\right)=(0,\mu_{\G}(x,y))+\F 1_{\mathcal{A}}$, 
for all $\xi, \eta$ in $\F$ and all $x,y$ in $\G$. 
Since $(\xi,x)+\F 1_{\mathcal{A}}=(0,x)+\F 1_{\mathcal{A}}$, 
the linear map 
$\psi:\mathcal{A}^{\prime} \to \G$, 
defined by $\psi\left((0,x)+\F 1_{\mathcal{A}}\right)=x$
gives the desired isomorphism between $(\mathcal{A}^{\prime},\mu^{\prime},T^{\prime})$ 
and $(\G,\mu_\G,L)$, and $\psi \circ T^{\prime}=L \circ \psi$.
\end{proof}

\medskip

\section{Invariant metrics on central extensions}

\medskip
To round up our results, we conclude as follows.
We have started
with a quadratic Hom-Lie algebra $(\g,\mu,T,B)$ having
$T$ in the centroid.
Besides, we have assumed right from the start that $\dim\Ker(T)>0$. 
Then, we showed how to produce a quadratic Lie algebra 
$(\g,[\,\cdot\,,\,\cdot\,],B)$, using actually the same $B$
and taking a vector space $V$ with $\dim V=\dim\Ker(T)$
we obtained a central extension $(\G=\g \oplus V,[\,\cdot\,,\,\cdot\,]_{\G})$
having an invariant metric $B_{\G}$ under which 
$V$ is isotropic. Now, our final result shows how from the
quadratic Lie algebra $(\G=\g \oplus V,[\,\cdot\,,\,\cdot\,]_{\G},B_{\G})$
which is a central extension of $(\g,[\,\cdot\,,\,\cdot\,])$
and with $V$ isotropic for its metric $B_{\G}$,
one may produce a quadratic Hom-Lie algebra $(\g,\mu,T,B)$
having $T$ in its centroid
and $\dim\Ker(T)=\dim V$.
In the course of the proof we shall use some results found in
\cite{GSV} which actually inspired us to address the problem
of the quadratic Hom-Lie algebras in the way we have presented it
in this work.

\medskip
\begin{Theorem}\label{teorema extensiones centrales}
{\sl Let $(\g,[\,\cdot\,,\,\cdot\,],B)$ be a quadratic Lie algebra
and let $V$ be an $r$-dimensional vector space. 
Let $(\G=\g \oplus V,[\,\cdot\,,\,\cdot\,]_{\G})$ be a central extension of 
$(\g,[\,\cdot\,,\,\cdot\,])$ by $V$. If $(\G,[\,\cdot\,,\,\cdot\,]_{\G})$ 
admits an invariant metric $B_{\G}$ where $V$ is isotropic, 
then there exists a quadratic Hom-Lie algebra $(\g,\mu,T,B)$ 
such that $[x,y]=T(\mu(x,y))=\mu(T(x),y)$, for all $x,y$ in $\g$
and $\dim\Ker(T)=r$\/.}
\end{Theorem}

\begin{proof}
Let $h=B^{\sharp}\circ \iota_{\g}^{\ast} \circ B_{\G}^{\flat}:\G \to \g$, 
and $k=B^{\sharp}_{\G}\circ \pi^{\ast}\circ B^{\flat}:\g \to \G$,
where $\pi:\G\to\g$ is the canonical projection,
and $\iota_{\g}:\g\to\G$ the inclusion map.
For each $x$ in $\G$, let $\rho(x):\G \to \g$ be defined by 
$\rho(x)(y)=h([x,y]_{\G})$, for all $y$ in $\g$. 
Since $\g\subset\G$, we may now
define $\mu:\g \times \g \to \g$ by $\mu(x,y)=\rho(x)(y)$,
and $T:\g \to \g$ by $T=\pi \circ k$. 
\textbf{Lemma\,2.2} of \cite{GSV}, proves that
$T \circ \rho(x)=\rho(x)\circ T=\ad(x)$,
for each $x$ in $\g$, which implies that 
$T(\mu(x,y))=\mu(T(x),y)=[x,y]$; that is,  
$T$ is a centroid.
Moreover, \textbf{Lemma\,2.2\,(vii)} of \cite{GSV} proves that
$\rho([x,y])=\rho(x) \circ T \circ \rho(y)-\rho(y) \circ T \circ \rho(x)$, 
which is equivalent to, $\mu(T(\mu(x,y)),z)=\mu(T(x),\mu(y,z))+\mu(T(y),\mu(z,x))$. 
Since $T$ is a centroid,
the last expression can be rewritten as,
$$
\mu(T(x),\mu(y,z))+\mu(T(y),\mu(z,x))+\mu(T(z),\mu(x,y))=0.
$$
Whence $(\g,\mu,T)$ is a Hom-Lie algebra 
with $T$ in its centroid.
In addition, \textbf{Lemma\,2.2} of \cite{GSV} also 
shows $T$ is $B$-self-adjoint and that $\rho$ actually defines
a map $\rho:\G \to \Der(\g)$, making $\rho(x)$ anti-self-adjoint
with respecto to $B$ for each $x$ in $\G$. Therefore,
$B(\mu(x,y),z)=B(\rho(x)(y),z)=-B(y,\rho(x)(z))=-B(y,\mu(x,z))$, 
which proves that $B$ is $\mu$-invariant and that 
$(\g,\mu,T,B)$ is a quadratic Hom-Lie algebra. 
Finally, since $V$ is isotropic, \textbf{Lemma\,2.7\,(ii)} of \cite{GSV}
proves that $\dim\Ker(T)=\dim V$.
\end{proof}

\medskip
\section{Example}

We shall now illustrate with a non-trivial example how
\textbf{Thm.\,\ref{teorema extensiones centrales}} works.
Let $(\g,[\,\cdot\,,\,\cdot\,],B)$ be the 6-dimensional $2$-step nilpotent
quadratic Lie algebra whose underlying space is
decomposed in the form
 $\g=\a \oplus {\mathfrak b}$, with
$\a=\operatorname{Span}\{a_1,a_2,a_3\}$ and
${\mathfrak b}=\operatorname{Span}\{b_1,b_2,b_3\}$.
Let $\sigma$ be the cyclic permutation $\sigma=(1\,2\,3)$ and define
the Lie bracket $[\,\cdot\,,\cdot\,]$ on $\g$ by letting
$[a_i,a_{\sigma(i)}]=b_{\sigma^2(i)}$, for all $1 \leq i \leq 3$,
and $[\g,{\mathfrak b}]=\{0\}$. The invariant metric
$B:\g \times \g \to \F$ is given by
$B(a_i,b_j)=\delta_{ij}$ ($1\le i,j\le 3$).
Now define the $B$-anti-self-adjoint derivations
$D_1,D_2,D_3\in\Der(\g)$, through,
$$
\aligned
D_i\left(a_{\sigma(i)}\right)& =-D_{\sigma(i)}(a_i)=a_{\sigma^2(i)}+b_{\sigma^2(i)},
\\
D_{i}\left(b_{\sigma(i)}\right) & =-D_{\sigma(i)}(b_i)=b_{\sigma^2(i)},
\\
D_i\left(a_i\right)& =D_i\left(b_i\right)=0,\,\,1 \leq i \leq 3.
\endaligned
$$
Observe that if these derivations were inner, their image
would lie in $\operatorname{Span}\{b_1,b_2,b_3\}
=[\g,\g]$, which is not the case.

\medskip
Let $\G=\g \oplus V$, where  
$V=\operatorname{Span}\{v_1,v_2,v_3 \}$ is a $3$-dimensional vector space.
Define the 2-cocycle $\theta:\g \times \g \rightarrow V$ by, 
$$
\theta(x,y)=B\left(D_1(x),y\right)v_1+B\left(D_2(x),y\right)v_2+B\left(D_3(x),y\right)v_3,
$$
for all $x$ and $y$ in $\g$. 
Thus, $\theta\left(a_i,a_{\sigma(i)}\right)=v_{\sigma^2(i)}$, for all $1 \leq i \leq 3$. 
Then use $\theta$ to define the Lie bracket $[\,\cdot\,,\,\cdot\,]_{\G}$ on 
$\G=\g \oplus V$ as in \eqref{expresion para corchete de extension central 2}; 
that is,
$$
\aligned
\,\left[a_i,a_{\sigma(i)}\right]_{\G}& = b_{\sigma^2(i)}+v_{\sigma^2(i)},\\
\,\left[a_i,b_{\sigma(i)}\right]_{\G}&=v_{\sigma^2(i)}, \quad \text{ for all }1 \leq i \leq 3.
\endaligned
$$
The $9$-dimensional Lie algebra on $\G$ thus defined is
a $3$-step nilpotent. We define the invariant metric 
$B_{\G}:\G \times \G \rightarrow \F$, by, 
$$
\aligned
& B_{\G}\left(b_i,b_j\right)=B_{\G}\left(a_i,v_j\right)=\delta_{ij},\\
& B_{\G}\left(b_i,a_j+v_k\right)=B_{\G}\left(a_i,a_j\right)=B_{\G}\left(v_i,v_j\right)=0,\,1 \leq i,j,k \leq 3.
\endaligned
$$
Clearly, $\G$ has the vector space decomposition, 
$\G=\a \oplus {\mathfrak b} \oplus V$, where both, $\a$ and $V$, are isotropic
subspaces under $B_{\G}$, and ${{\mathfrak b}}^{\perp}=\a\oplus V$. 

\medskip
We shall now describe the Hom-Lie algebra
$(\g,\mu,T)$ of \textbf{Thm.\,\ref{teorema extensiones centrales}}.
The linear map $h:\G \to \g$ from \eqref{definicion h}
is given in this example by, 
$h(a_{j})=0$, $h(b_j)=a_{j}$ and $h(v_j)=b_j$, for all $1 \leq j \leq 3$.

Now, \textbf{Thm.\,\ref{teorema extensiones centrales}}
says that the map $\mu:\g\times\g\to\g$ is defined in terms of $h$
and the Lie bracket $[\,\cdot\,,\,\cdot\,]_{\G}$ as folows:
$\mu(x,y)=h([x,y]_{\G})$ for all $x$ and $y$ in $\g$. 
Since $V \subset C(\g)$, for any $x,y$ in $\g$ and $u,v$ in $V$, 
we may define $\mu_{\G}:\G \times \G \to \G$ 
by, $\mu_{\G}(x+u,y+v)=\mu(x,y)$. 

\medskip
Observe that ${\mathfrak b}=C(\g)$ is $B$-isotropic and 
$D_j({\mathfrak b}) \subset {\mathfrak b}$ for all $j$. Thus, 
$\mu({\mathfrak b},{\mathfrak b})=h\left([{\mathfrak b},{\mathfrak b}]_{\G}\right)=\{0\}$. 
Therefore, for all $1 \leq j,k \leq 3$ and for all $x$ in $\g$, 
the Hom-Lie bracket $\mu_{\G}:\G \times \G \to \G$ is given by,
$$
\aligned
\mu_{\G}(a_j,a_{\sigma(j)})&=h\left(\left[a_j,a_{\sigma(j)}\right]_{\G}\right)=D_j\left(a_{\sigma(j)}\right)=a_{\sigma^2(j)}+b_{\sigma^2(j)},\\
\mu_{\G}\left(a_j,b_{\sigma(j)}\right)&=h\left(\left[a_j,b_{\sigma(j)}\right]_{\G}\right)=D_j(b_{\sigma(j)})=b_{\sigma^2(j)},\\
\mu_{\G}\left(b_j,b_k\right)&=\mu_{\G}\left(v_j,x\right)=h\left(\left[v_j,x\right]_{\G}\right)=0.\\
\endaligned
$$
The linear map $T:\g\to\g$ that works as a twist map 
for $\mu:\g \times \g \to \g$ and completes the description of the 
$6$-dimensional Hom-Lie algebra $(\g,\mu,T,)$ is given by, 
$T(a_j)=b_j$ and $T(b_j)=0$, for all $1 \leq j \leq 3$.
 Then, the linear map $k:\g \to \G$ defined in \eqref{definicion de k}
is given by, $k(a_j)=b_j$ and $k(b_j)=v_j$, for all $1 \leq j \leq 3$. 

\medskip
The invariance of $B$ under the Hom-Lie bracket $\mu$ of $\g$,
is obtained from the fact that the derivations 
$D_1$, $D_2$ and $D_3$ are $B$-anti-self-adjoint 
and from the fact that ${\mathfrak b}$ is $\mu$-Abelian.
Thus, for any $x$ and $y$ in $\g$, we have, 
$B\left(\mu(a_j,x),y\right)=B\left(D_j(x),y\right)=-B\left(x,D_j(y)\right)=-B\left(x,\mu(a_j,x)\right)$.
On the other hand, the linear map $L:\G\to\G$
that defines the $9$-dimensional Hom-Lie algebra $\left(\G,\mu_{\G},L\right)$
of \textbf{Prop.\,\ref{teorema nuevo 3}} obtained 
as in \textbf{Thm.\,\ref{teorema extensiones centrales}}, is given by,
$L(x+v)=k(x)$ for all $x$ in $\g$ and all $v$ in $V$.
Then, $L(a_j) =b_j$, $L(b_j)=v_j$ and $L(v_j)=0$, for all $1 \leq k \leq 3$. 

\medskip
We shall describe the product $(x,y) \mapsto xy$ of \textbf{Prop.\,\ref{proposicion conexion}}
in terms of the basis $\{\,a_j,b_k,v_{\ell}\mid\,1 \le j,k, \ell \le 3\,\}$ of $\G$. 
It is shown there that, $2xy=\mu_{\G}(x,y)-[h(x),y]_{\G}+[x,h(y)]_{\G}$, 
for all $x $ and $y$ in $\G$. 
Then, $x\,u=u\,x=0$ for all $x$ in $\G$ and $u$ in $V$. 
Since $\mu_{\G}\left(a_j,\cdot\right)\vert_{\g}=D_j$ and $h(a_j)=0$ for all $j$, 
it follows that, $2a_jx=D_j(x)+\left[a_j,h(x)\right]_{\G}$ for all $x$ in $\g$. 
In particular, $2a_ja_k=D_j(a_k)=-D_k(a_j)=-2a_ka_j$ for all $1\le j,k\le 3$. 
On the other hand, $2a_1b_2=2b_3+v_3$, $2a_2b_3=2b_1+v_1$ and $2a_3b_1=2b_2+v_2$. 
Analogously, $2b_1a_2=v_3$, $2b_2a_3=v_1$ and $2b_3a_1=v_2$. 
Finally, this product on ${\mathfrak b}$ can be obtained as follows: 
Since $h(b_j)=a_j$, for all $1 \leq j \leq 3$, we get, 
$2b_jb_k=-\left[h(b_j),b_k\right]_{\G}+\left[b_j,h(b_k)\right]_{\G}
=-\left[a_j,b_k\right]_{\G}-\left[a_k,b_j\right]_{\G}$, for all $1 \leq j,k \leq 3$. 
Now, using again the fact that the derivations $D_j$ are
$B$-anti-self-adjoint,
it is not difficult to see that 
$\left[a_j,b_k\right]_{\G}=-\left[a_k,b_j\right]_{\G}$.
Therefore, $b_jb_k=0$, for all $1\le j,k\le 3$.

\medskip
\section*{Acknowledgements} 
The authors acknowledge the support received
through CONACYT Grant $\#$ A1-S-45886. The author RGD thanks the support 
provided by CONACYT post-doctoral fellowship 769309. Finally, GS acknowledges the
support provided by PROMEP grant UASLP-CA-228
and ASV acknowledges the support given by MB1411.

\end{document}